\documentclass[12pt,a4paper]{amsart}
\usepackage{amsmath,amssymb,amsthm}

\usepackage[margin=1in,footskip=0.25in]{geometry}

\usepackage[breaklinks=true,colorlinks=true,linkcolor=blue,citecolor=red,urlcolor=green]{hyperref}
\usepackage{graphicx,psfrag}
\newtheorem{Thm}{Theorem}[section]
\newtheorem{Lem}[Thm]{Lemma}
\newtheorem{Pro}[Thm]{Proposition}
\newtheorem{Cor}[Thm]{Corollary}

\newtheorem*{Con*}{Conjecture}

\theoremstyle{definition}

\theoremstyle{remark}

\newcommand{\la}{\lambda}

\renewcommand{\phi}{\varphi}

\newcommand{\Ric}{\operatorname{Ric}}
\newcommand{\tr}{\operatorname{tr}}

\renewcommand{\cos}{\operatorname{cos}}

\begin{document}

\title[]
{}

\title
[Asymptotically harmonic manifolds of dimension $3$]
{Asymptotically harmonic manifolds of dimension \boldmath$3$\unboldmath \; with minimal horospheres}

\author{Jihun Kim\and JeongHyeong Park\and Hemangi  Madhusudan Shah}
\address{Department of Mathematics, Sungkyunkwan University, Suwon, 16419, Korea}
\email{jihunkim@skku.edu, parkj@skku.edu}
\address{Harish Chandra Research Institute,  A CI of Homi Bhabha National Institute,
 Chhatnag Road, Jhusi, Allahabad 211019, India}
\email{hemangimshah@hri.res.in}

% \date{Aug 31, 2023}

% \title {Asymptotically harmonic manifolds of dimension $3$ with minimal horospheres}
% \author{Jihun Kim\and JeongHyeong Park\and Hemangi Shah}
% \address{Department of Mathematics, Sungkyunkwan University, Suwon, 16419, Korea}
% \email{jihunkim@skku.edu, parkj@skku.edu}
% \address{Harish Chandra Research Institute, Chhatnag Road, Jhusi, Allahabad 211019, India}
% \email{hemangimshah@hri.res.in}
% \date{Aug 3, 2023}
% \author {Jihun Kim, JeongHyeong Park, and Hemangi Shah \\\\
% \newline{Email: {\tt jihunkim@skku.edu},{
% \tt parkj@skku.edu}, {\tt hemangimshah@hri.res.in.}}}
% \maketitle {Harish Chandra Research Institute, Chhatnag Road, Jhusi, Allahabad -\\
% \indent
% 211019, India.}
 \begin{abstract}
% Let $(M,g)$ be a complete, simply connected Riemannian
% manifold of dimension $3$ without conjugate points. We show that
% if  $M$ is asymptotically harmonic of  constant $h = 0$,  then
% $M$ is a flat manifold. This theorem shows that any
% asymptotically harmonic manifold in dimension $3$ is a
% symmetric space,  thus completing the classification
% of asymptotically harmonic manifolds in dimension $3$.
In \cite{S2}, it was shown that,  if $M$ is a $3$-dimensional asymptotically harmonic with minimal horospheres, then $M$ is flat.  However, there is a gap in the proof of this paper.  In this paper, we provide the correct proof of the result.
Thus we complete the classification
of asymptotically harmonic manifolds of dimension $3$:
An asymptotically harmonic manifold of dimension $3$ is either a flat or real hyperbolic space.
\end{abstract}

\dedicatory{Dedicated to the memory of Professor Joseph A. Wolf who sadly passed away on August 14, 2023.}

\maketitle

% {{\bf Key Words} : Asymptotically harmonic manifold, Busemann
% function,\\ \indent ~~~~~~~~~~~~~~mean curvature, horospheres.}
% \\\\
% \noindent Mathematics Subject Classification (2000): Primary
% 53C35; Secondary 53C25.

\section{Introduction}

 Let $(M,g)$ be a complete, simply connected Riemannian
manifold without conjugate points.  We denote the unit tangent
bundle of $M$ by $SM$. For $v \in SM$, let $\gamma_{v}$ be the geodesic
with ${\gamma'_{v}}(0) = v$ and $b_{v} (x) =\displaystyle
\lim_{t\to \infty} (d (x,\gamma_v(t)) - t)$, the corresponding
\emph{Busemann function} for $\gamma_{v}$. The level sets
of the Busemann function are called \emph{horospheres} of $M$.
% The level sets
% ${b_v}^{-1}(t)$ are called {\it horospheres} of $M$.

 A complete, simply connected Riemannian manifold without conjugate points is called \emph{asymptotically harmonic} if the mean curvature of its horospheres is a universal constant, that is, if
its Busemann functions satisfy $\Delta b_v \equiv h,\; \forall v
\in SM$, where $h$ is a nonnegative constant. Then $b_v$ is a
smooth function on $M$ for all
 $v$ and all the  horospheres of $M$ are smooth, simply connected
hypersurfaces in $M$ with constant mean curvature $h$.

On the other hand, a Riemannian manifold is called (locally) \emph{harmonic}, if about any point all the geodesic spheres of sufficiently small radii are of constant mean curvature. Since a harmonic manifold is Einstein, harmonic manifolds of dimensions $2$ and $3$ are of constant sectional curvature.  In 1944, Lichnerowicz \cite{Li} showed that a $4$-dimensional harmonic manifold is locally symmetric and conjectured that a harmonic manifold is flat or {locally} rank-one symmetric space (this conjecture is called \emph{Lichnerowicz's conjecture}). Nikolayevsky \cite{Ni} proved Lichnerowicz's conjecture in dimension $5$. Szab\'o \cite{Sz} proved the conjecture for compact harmonic manifolds.
On the other hand, Damek and Ricci \cite{DR} constructed nonsymmetric harmonic manifolds of dimension $\ge 7$, which are called \emph{Damek-Ricci spaces}.

It follows from \cite{RS1} that every complete, simply connected harmonic
manifold without conjugate points is asymptotically harmonic. It is natural to ask whether asymptotically harmonic manifolds are locally symmetric?  Heber \cite{He} proved that for noncompact, simply connected homogeneous space, the manifold is asymptotically harmonic and Einstein if and only if it is flat, or rank-one symmetric space of noncompact type, or a nonsymmetric Damek-Ricci space. For more characterizations of asymptotically harmonic manifolds, we refer to \cite{KP, Z}. {By the Riccati equation \eqref{Riccati}, one can easily check that a Ricci-flat asymptotically harmonic manifold is also flat.}

In \cite{RS}, it was shown that harmonic manifolds with minimal horospheres are flat.
In \cite{S2},  it was shown that an asymptotically harmonic manifold of dimension $3$ with minimal horospheres is flat.
But we found that there is some gap in the proof of \cite[Lemma 2.2]{S2}: $\tr\sqrt{-R(x,v)v}=0$ {\it does not} imply $R(x,v)v=0$.
In this article,  we give the correct proof of this theorem.

\vspace{0.15in}

Now we recall some notations about asymptotically harmonic manifolds (see \cite{H, S, S2}).

For $v\in SM$ and $x\in v^\perp$, we define $u^{\pm}(v)\in {\rm End}(v^\perp)$ by
\begin{equation*}
    u^+(v)(x)=\nabla_x \nabla b_{-v},\quad u^{-}(v)(x)=-\nabla_x\nabla b_{v}.
\end{equation*}
Also, $u^{\pm}$ satisfies the Riccati equation along the orbits of the geodesic flow $\varphi^t: SM\to SM$. If we let $u^{\pm}(t):=u^{\pm}(\varphi^t v)$ and $R(t):=R\big(~\cdot~, \gamma_v'(t)\big)\gamma_v'(t) \in {\rm End}(\gamma_v'(t)^\perp)$, then $u^{\pm}(t)$ satisfy the following Riccati equation:
\begin{equation}\label{Riccati}
    (u^{\pm})' + (u^{\pm})^2 + R = 0.
\end{equation}
Here, $u^{+}(t)$ and $u^{-}(t)$ are called as \emph{unstable} and \emph{stable} Riccati solution, repsectively.
We have $\tr u^+ (v) =\Delta b_{-v} = h$ and $\tr u^{-}(v) = -\Delta b_v = -h$ for all $v\in SM$. \\

Clearly by the Riccati equation~\eqref{Riccati}, we see that any $2$-dimensional asymptotically harmonic manifold,
is either a flat or a real hyperbolic plane of constant curvature $-h^2$.
This shows that the study of asymptotically harmonic manifold  {begins} with dimension $3$.
 Towards this, it was shown in  \cite{H} that any Hadamard
 asymptotically harmonic manifold of bounded sectional curvature, satisfying  {some} mild hypothesis
 on  {the}  curvature tensor is a real hyperbolic space of constant sectional curvature $-\frac{h^2}{4}$.
 Finally, this result was improved by Schroeder and Shah \cite{S} by relaxing the hypothesis on the curvature tensor.

 \begin{Thm}[\cite{S}]\label{thm:ross}
    If $(M,g)$ is a $3$-dimensional asymptotically harmonic manifold with $h>0$, then $M$ is a real hyperbolic space of constant sectional curvature $-\frac{h^2}{4}$.
\end{Thm}

%\begin{Pro}
 %   If $(M,g)$ is a $2$-dimensional asymptotically harmonic manifold, then $M$ is either a Euclidean surface or a %real hyperbolic surface with constant curvature $-h^2$.
%\begin{proof}
 %   Since $\dim M=2$, we obtain $u^{+}(v)(x)=h x$ for all $v\in SM$ and $x\in v^\perp$. Then the Riccati %equation~\eqref{Riccati} yields $h^2 x =u^{+}(v)^2(x) = -R(x,v)v$. It follows that $M$ is a Euclidean surface if %
 % $h=0$, and $M$ is a real hyperbolic space $\mathbb{H}^2 (-h^2)$ if $h>0$.
%\end{proof}
%\end{Pro}

%\begin{Thm}[\cite{H}]
  %  Let $(M,g)$ be a Hadamard {\rm(}i.e. complete, simply connected manifold with nonpositive sectional %curvature{\rm)} of dimension $3$. Suppose that the sectional curvature satisfies $-b^2 \le K \le -a^2 <0$ and the %curvature tensor satisfies $||\nabla R||\le C$ for suitable constants $0<a\le b$ and $C\ge 0$.
  %  If $M$ is asymptotically harmonic, then $M$ is of constant sectional curvature.
%\end{Thm}

%After then, Schroeder and Shah \cite{S} showed the same result without any hypothesis on the curvature tensor:
%\begin{Thm}[\cite{S}]\label{thm:ross}
 %   If $(M,g)$ is a $3$-dimensional asymptotically harmonic manifold with $h>0$, then $M$ is a real hyperbolic %space of constant sectional curvature $-\frac{h^2}{4}$.
%\end{Thm}

Our main result is the following.
\begin{Thm}\label{thm:mainthm}
If $(M,g)$ is a $3$-dimensional asymptotically harmonic manifold with minimal horospheres {\rm(}i.e., $h=0${\rm)}, then $M$ is flat.
\end{Thm}

% For more details on this subject we refer to the
% discussion and to the references in \cite{H}. Important results in
% this context are contained in \cite{BCG}, \cite{K}, and \cite{S}.
% In \cite{S} the following result was proved.

% \begin{Thm}\label{thm:ross}
% Let $(M,g)$ be a complete, simply connected Riemannian manifold of
% dimension $3$ without conjugate points. We show that $M$ is a
% hyperbolic manifold of constant sectional curvature
% $\frac{-h^2}{4}$, provided $M$ is asymptotically harmonic of
% constant $h > 0$.
% \end{Thm}

% Note that the strategy of the proof of the above theorem strictly works
% when  $h > 0$.  This paper extends the above result for $h = 0$.

% \begin{Thm}\label{thm:mainthm}
%  Let $(M,g)$ be a complete, simply connected Riemannian manifold of
% dimension $3$ without conjugate points. If $M$ is asymptotically
% harmonic of constant $h = 0$, then $M$ is flat.
% \end{Thm}
% \indent

From Theorem \ref{thm:ross} and Theorem  \ref{thm:mainthm}, we obtain the complete classification of asymptotically harmonic manifolds of dimension $3$.
% together classify  asymptotically harmonic  manifolds in dimension
% $3$.

\begin{Thm}\label{class}
Let $(M,g)$ be a $3$-dimensional asymptotically harmonic manifold. Then $M$ is flat if $h=0$, and $M$ is a real hyperbolic space of constant sectional curvature $-\frac{h^2}{4}$ if $h>0$.
% Asymptotically harmonic  manifold of  constant $h$ in dimension
% $3$  is  a symmetric  space. If  $h = 0$, then it is a flat  space and
% if $h > 0$, then it  is  a  hyperbolic space of  constant curvature  $\frac{-h^2}{4}$.
\end{Thm}

\section{Proof of Theorem \ref{thm:mainthm} }

% \indent
% \;\;In  this  section we  prove the main theorem viz.
%  Theorem \ref{thm:mainthm} as stated
% in the introduction.  We use ideas in  my paper
% with Schroeder  \cite{S}  to prove the main theorem.
% The proof of  main result (Theorem \ref{lem:curvature})
% follows closely  the result  in \cite{S}.\\
% \indent
% First we recall some notations which were already used in the paper \cite{S}.\\
% Our general
% assumption is that $M$ is $3$-dimensional, has no conjugate points
% and is asymptotically harmonic with constant $h = 0$. \\
% Let $SM$ denote unit tangent bundle of $M$.
% For $v \in SM$ and $x \in v^{\perp},$ let
% $$ u^{+}(v)(x) = \nabla_x \nabla b_{-v} \ \ \ \mbox{and}\ \ \
% u^{-}(v)(x) = - \nabla_x \nabla b_{v}.$$ Thus $u^{\pm}(v) \in
% \mbox{End} \;(v^{\perp}).$ The endomorphism fields $u^{\pm}$
% satisfy the Riccati equation along the orbits
% of the geodesic flow ${\phi}^{t} : SM \rightarrow SM$.\\
% Thus if
% $u^{\pm}(t):=u^{\pm}(\phi^tv)$ and
% $R(t):=R(\cdot,\gamma_v'(t))\gamma_v'(t)\in\mbox{End}(\gamma_v'(t)^{\perp}),$
% then
% \begin{eqnarray}\label{Riccati}
%  (u^{\pm})' + (u^{\pm})^2 + R = 0.
%  \end{eqnarray}

% We will be using  the following result proved in \cite{Z}  to prove
% our main result.
Let $(M,g)$ be a complete, simply connected Riemannian
manifold without conjugate points.
To prove our main result, we need the following {results}.

\begin{Lem}[\cite{Z}]\label{cont}
If $(M, g)$ is an asymptotically harmonic
manifold, then the map $v \mapsto u^{\pm}(v)$ is continuous on $SM$.
\end{Lem}

\begin{Lem}\label{vec}
    Let $(M,g)$ be an $n$-dimensional asymptotically harmonic manifold. Then we have $\Ric(v,v)\le -\frac{h^2}{n-1}$. In particular, if $M$ has minimal horospheres, then $\Ric(v,v)\le 0$.
\begin{proof}
    From the Riccati equation~\eqref{Riccati}, we get $\Ric(v,v)=-\tr(u^{+}(v)^2)$. For the eigenvalues of $u^{+}(v)$, the arithmetic-quadratic mean inequality yields $h^2 = (\tr u^{+}(v))^2\le (n-1)\tr(u^{+}(v)^2)$, and hence we have $\Ric(v,v)\le -\frac{h^2}{n-1}$. Note that if $M$ has a minimal horosphere (i.e., $h=0$), then $\Ric(v,v)\le 0$.
\end{proof}
\end{Lem}

\begin{Lem}\label{lem:curvdiag}
    If $(M,g)$ is a $3$-dimensional asymptotically harmonic manifold with minimal horospheres, then for every $p\in M$, there exists an orthonormal basis $\{e_1,e_2,e_3\}$ of $T_p M$ such that the curvature operator
    \begin{equation*}
        \mathcal{R}:\wedge^2 T_pM \to \wedge^2 T_p M \;{\textrm{given by}}\;\ g(R(X\wedge Y),Z\wedge W) = g(R(X,Y)W,Z)
    \end{equation*}
    is diagonal.
    \begin{proof}
    % Since the eigenvalues of $u^{+}(v)$ vary continuously and $\Tr u^{+}(v) = h =0$, we denote the eigenvalues of $u^{+}(v)$ by $\la(v)$ and $-\la(v)$.
    As $\dim M =3$, we can identify $S_p M$ with the standard $2$-sphere ${S}^2$, and $v^\perp$ with $T_v S^2$ for $v\in S_p M = {S}^2$. Eigenvalues of  $u^{+}(v)$ are  $\la(v)$ and $-\la(v)$.  As $TS^2$ is nontrivial,  an easy
  topological argument shows that for every $p\in M$, there exists $v_0 \in S_p M$ such that $u^{+}(v_0 ) =0$ (see \cite[p. 848]{H}).

    Let $v_0 = e_1$. Then $u^{+}(e_1)'(x) = -R(x,e_1)e_1$ for $x\in \{e_1\}^\perp$. Since  $\tr u^{+}(v) = 0$ for all $v \in SM$, it follows that  $\tr u^{+}(v)' =0$. {Consequently, we have that Ricci$(e_1, e_1) = 0$  and
 thus {\it Ricci curvature attains maximum}. }

   Denote the eigenvalues of $u^{+}(e_1)'$ by $\mu$ and $-\mu$. Without loss of generality, we assume that $\mu \ge 0$.
    If we let $\{e_2, e_3\}\in \{e_1\}^\perp$ be the eigenvectors of $u^{+}(e_1)'$, we have $R(e_2,e_1)e_1 = -\mu e_2$ and $R(e_3,e_1)e_1 = \mu e_3$. We also have  $\Ric(e_3,e_3)\le 0$ by Lemma \ref{vec}. Hence $K(e_2,e_3):=g(R(e_2,e_3)e_3,e_2)= - \eta \le -\mu \leq 0$.

    Consider for $t\in (-\varepsilon,\varepsilon)$ the vectors $v_t=\cos te_1 + \sin te_2$. Then we let
    \begin{equation*}
        f(t):=\Ric(v_t,v_t)=K(e_1,e_2)+\sin^2 t K(e_2,e_3) + \cos^2 t K(e_1,e_3) + \sin 2t\, g(R(e_1,e_3)e_3,e_2).
    \end{equation*}
    Since $f(0)=0$ and it is maximal, $f'(0)=0$. This implies $g(R(e_1,e_3)e_3,e_2)=0$. If we replace $e_2$ with $e_3$ in the above computation, we also obtain  $g(R(e_1,e_2)e_2,e_3)=0$. Therefore, for every point $p$, the curvature operator is diagonal and it is given by
    \begin{equation*}
        {\mathcal{R}}(p) =
        \begin{pmatrix}
           - \mu(p) & 0 & 0\\
            0 & \mu(p) & 0\\
            0 & 0 & { -\eta(p)}
        \end{pmatrix}
    \end{equation*}
    in the basis $\{e_1,e_2,e_3\}$ of $T_p M$.
  \end{proof}
\end{Lem}

\noindent
{{\bf Remark:} In the above lemma,  from the construction of $\mu$ and $\eta$ (using $\tr u^{+}(v)'=0$ and $\Ric(e_3,e_3)\le 0$), we note that the eigenvalues $\mu$
 and $\eta$ can be interchanged, but the aforementioned form of the curvature operator is the same.}

%  \begin{Lem}\label{vec}
%  For every $p  \in  M$,  there exists $v \in  S_pM$  such  that
% ${u^{+}}'(v) = 0$.  Consequently, it follows that  $R(x,v)v = 0$.
% In particular, Ricci$(v, v) = 0$.
% \end{Lem}
% \begin{proof}
%  As $M$ is asymptotically harmonic
% of constant $h =0$, $\tr u^{+}(v) = 0$.
% Hence,  $\tr {u^{+}}'(v) = 0$. Eigen values of  $u^{+}(v)$  viz. $\lambda^{+}(v)$,
% $-\lambda^{+}(v)$  vary continuously with $v$,
% by Theorem \ref{cont}. Since $TS^{2}$ is nontrivial, an easy topological argument shows
% that for every $p \in M$ there exists $v \in S_pM$ such that the two eigenvalues
% of  ${u^{+}}'(v)$  coincide.  Hence,  ${u^{+}}'(v) = 0$.  Now  Riccati equation shows that
% $(u^{+})^2(v)(x) + R(x,v)v = 0.$  Hence, $-R(x,v)v$ and consequently $\sqrt {-R(x,v)v}$
% is  a  non-negative operator.  But, $u^{+}(v)(x)  =  \pm \sqrt {-R(x,v)v}$.  Hence, $ \tr \sqrt {-R(x,v)v} = 0 $
%  Consequently, it follows that  $R(x,v)v = 0$.  In particular, Ricci$(v, v) = 0$.
 % \end{proof}
% \begin{Lem} \label{lem:estimatericci}
%  For all $v \in SM$ we have Ricci$(v, v) \leq 0.$
%  \end{Lem}
% \begin{proof} The Riccati equation for $t\mapsto u^{+}(t)$ implies
%  $(u^{+})' + (u^{+})^2 + R = 0.$ Hence $\tr(u^{+})^2 + \tr R = 0.$
%  Thus Ricci$(v, v) = -({{\la}_1}^{2}(v) +{{\la}_2}^{2}(v)).$
% Consequently, Ricci$(v, v) \leq  0.$
% \end{proof}

{
 \begin{Pro}
Let $(M^3, g)$ be an asymptotically harmonic manifold with minimal horospheres. Then
around  any point $p \in M$,  we can find an  orthogonal coordinate system $(U, (x^1, x^2, x^3))$
in which the metric is diagonal, that is $g=\sum\limits_{i=1}^{3}a_i(dx^{i})^2$ for some positive function $a_i$.
And the orthonormal basis  $\{e_1,e_2,e_3\}$ of $T_{p}M$ obtained in Lemma~\ref{lem:curvdiag}
with  the associated orthonormal frame with  $e_i = \dfrac{1}{\sqrt{a_i}}\dfrac{\partial}{\partial x^i} (p)$ for $i=1,2,3$.
\end{Pro}
\begin{proof}
The existence of a  required orthogonal coordinate system along with {the} associated orthonormal frame  $\{e_1,e_2,e_3\}$ of $T_{p}M$
is guaranteed  by Theorem $4.2$ of
\cite{DY}, which says that {\it every $3$-dimensional Riemannian manifold has orthogonal coordinates, that is, coordinates in which the Riemannian metric has a diagonal form}.
\end{proof}}

%\textcolor{red}{
%\begin{Lem}[\cite{DY}]\label{lem:orthocoord}
 %   Every $3$-dimensional Riemannian manifold has orthogonal coordinates, that is, coordinates in which the  %Riemannian metric has a diagonal form.
%\end{Lem}
%}
%\textcolor{red}{
%As a result, we can choose suitable orthogonal coordinates $(x^1, x^2, x^3)$ satisfying the metric
 %$g=\sum\limits_{i=1}^{3}a_i(dx^{i})^2$ for some positive function $a_i$, from which $\{e_1,e_2,e_3\}$ of
 %Lemma~\ref{lem:curvdiag} is the associated orthonormal frame with $e_i = \dfrac{1}{\sqrt{a_i}}\dfrac{\partial}
  %{\partial x^i}$ for $i=1,2,3$.
%}

\begin{Cor}\label{cor}
 Let  $\{e_1, e_2, e_3\}$ be an orthonormal basis of $T_p M$ obtained from Lemma~\ref{lem:curvdiag}.
 % from the above lemma.
 Then the  Christoffel symbols $\Gamma_{ij}^k=0$ for all distinct indices $i,j,k$.
  \end{Cor}
  \begin{proof}
    By the Koszul formula, the Christoffel symbol $\Gamma_{ij}^{k}$ for $\{e_1, e_2, e_3\}$ is given by
    \begin{equation*}
        \Gamma_{ij}^{k} = \frac{1}{2}\big(g([e_i,e_j],e_k)-g([e_i,e_k],e_j)-g([e_j,e_k],e_i)\big).
    \end{equation*}
    Since $e_i = \dfrac{1}{\sqrt{a_i}}\dfrac{\partial}{\partial x^i}$ for orthogonal coordinates $(x^i)$, we have that $\Gamma_{ij}^{k}=0$ for all distinct indices $i,j,k$.
  \end{proof}

% \begin{proof}
% This follows from Theorem $3.3$ of \cite{M}, which says that {\it in a $3$-dimensional manifold
% a metric is a metric of diagonal curvature if and only if this metric and its Ricci tensor are diagonalized simultaneously}.
% Thus $g_{ij}(p) = g_{i}(p) \delta_{ij}$, consequently  $\Gamma_{ij}^k(p)=0$ for all distinct indices $i,j,k$.
% \end{proof}

\begin{Lem}\label{lem:curvature}
 If $(M,g)$ is a $3$-dimensional asymptotically harmonic manifold with minimal horospheres, then the sectional curvature $K$ of $M$ satisfies $K \le 0$.
\end{Lem}
\begin{proof}
Let  $\{e_1,e_2,e_3\}$  be as in the  Lemma \ref{lem:curvdiag}. At fixed point $p\in M$,  we can consider a geodesic $\gamma_{e_1}$ with $\gamma_{e_1}(0)=p$, $\gamma_{e_1}'(0)=e_1$ and $b_{e_1}(x)=\displaystyle\lim_{t\to \infty}(d(x,\gamma_{e_1}(t))-t)$. Since $\nabla_{e_1}e_1 = 0$, we get Christoffel symbols $\Gamma_{11}^{1}=\Gamma_{11}^{2}=\Gamma_{11}^{3}=0$. By {Corollary} \ref{cor}, $\Gamma_{ij}^k =0$ for all distinct indices $i,j,k$. So, we need to compute  {the} other Christoffel symbols when at least two indices are the same. From $g(\nabla_{e_i}e_j , e_k) = -g(e_j, \nabla_{e_i} e_k)$ for all $i,j,k$, we get $\Gamma_{ij}^k = -\Gamma_{ik}^j$.

\smallskip
\noindent\textsc{Case 1}: All indices are the same.  If $1\le i=j=k \le 3$, we have $\Gamma_{11}^1 = \Gamma_{22}^2 = \Gamma_{33}^3=0$.

\smallskip
\noindent\textsc{Case 2}: Two indices are the same. From $g(\nabla_{e_i}e_j , e_j)=0$, we obtain $\Gamma_{ij}^j=0$.

Therefore, all the Christoffel symbols are zero except $\Gamma_{21}^2=-\Gamma_{22}^1$, $\Gamma_{23}^2=-\Gamma_{22}^3$, $\Gamma_{31}^3=-\Gamma_{33}^1$, $\Gamma_{32}^3=-\Gamma_{33}^2$.

Now we compute the curvature tensor $R(e_2,e_1)e_1$. Since $\nabla_{e_2}e_1 = \Gamma_{21}^2 e_2$ and $[e_2,e_1] = \nabla_{e_2}e_1 - \nabla_{e_1}e_2 = \Gamma_{21}^2 e_2$, we obtain
\begin{align*}
    -\mu e_2 &= R(e_2,e_1)e_1\\
    &=\nabla_{e_2}\nabla_{e_1}e_1 - \nabla_{e_1}\nabla_{e_2}e_1 - \nabla_{[e_2,e_1]}e_1\\
    &=-\left(e_1 \Gamma_{21}^2 +\left(\Gamma_{21}^2\right)^2\right)e_2.
\end{align*}
Similarly, we also obtain $\mu e_3  =-\left(e_1 \Gamma_{31}^3 +\left(\Gamma_{31}^3\right)^2 \right)e_3$,
as  $\mu e_3 = R(e_3 ,e_1)e_1$.
Hence,  we have
\begin{equation*}
    -\mu = -\left(e_1 \Gamma_{21}^2 +\left(\Gamma_{21}^2\right)^2\right),\quad \mu=-\left(e_1 \Gamma_{31}^3 +\left(\Gamma_{31}^3\right)^2 \right),
\end{equation*}
and hence consequently,

\begin{equation}\label{sum-chri}
    0 = e_1 \left(\Gamma_{21}^2  +  \Gamma_{31}^3\right) +\left(\Gamma_{21}^2\right)^2+\left(\Gamma_{31}^3\right)^2.
\end{equation}
We observe that
\begin{align*}
    \Gamma_{21}^2  +  \Gamma_{31}^3 &= \Gamma_{11}^1 + \Gamma_{21}^2  +  \Gamma_{31}^3 \\
    &=g(\nabla_{e_1}e_1, e_1)+g(\nabla_{e_2}e_1, e_2)+g(\nabla_{e_3}e_1, e_3)\\
    &=\operatorname{div}(e_1)(p) = \operatorname{div}(\nabla b_{-e_1})(p) = \Delta b_{-e_1}(p) = 0.
\end{align*}
 From  (\ref{sum-chri})  this yields that  $\Gamma_{21}^2  = \Gamma_{31}^3=0$, so $K(e_1 ,e_2)=K(e_1, e_3)=0$ and $K(e_2, e_3) \le 0$. This implies that the sectional curvature $K$ of $M$ satisfies $K\le 0$.
\end{proof}

Now finally, we  prove our main result.

\medskip
\noindent{\textit{Proof of Theorem \ref{thm:mainthm}}}. From Lemma \ref{lem:curvature}, the sectional curvature $K\le 0$. Using standard comparison geometry, we obtain two eigenvalues $\la_1 (v)$ and $\la_2 (v)$ of the operator $u^{+}(v)$ satisfy $\la_1(v), \la_2 (v)\ge 0$ for all $v\in SM$. But $\la_1(v)+\la_2 (v)=\tr u^{+}(v) = h = 0$, we have $\la_1(v)=\la_2 (v)=0$. Therefore, $u^{+}(v)\equiv 0$  and so $R(x,v)v = 0$ for all $x\in v^\perp$ by  Riccati equation  (\ref{Riccati}).
This completes the proof.\hfill \qed

\section{Concluding Remarks}
Let $(M,g)$ be a complete, simply connected Riemannian
manifold without conjugate points. In this section, the dimension of $M$ is arbitrary.
We define $V(v):=u^{+}(v)-u^{-}(v)$ and $X(v):=-\frac{1}{2}(u^{+}(v)+u^{-}(v))$, correspondingly $V(t):=V(\varphi^t v)$ and $X(t):=X(\varphi^t v)$, respectively, where $\varphi^t :SM \to SM$ is the geodesic flow. Then the Riccati equation~\eqref{Riccati} gives  $V' = XV+VX$.

If we assume that $M$ is an asymptotically harmonic manifold with minimal horospheres and Einstein, then we have that $M$ is flat as follows:

\begin{Lem}[\cite{Z}]\label{lem:nonnegativeV}
    Let $(M,g)$ be an asymptotically harmonic manifold. Then $V(v)\ge 0$ for all $v \in SM$.
\end{Lem}

\begin{Lem}[\cite{S1}]\label{lem:u+}
    If $(M,g)$ is an asymptotically harmonic manifold, then for every point $p\in M$, there exists $v_0 \in S_p M$ such that $u^{+}(v_0)$ and $u^{+}(-v_0)$ have the same eigenvalues. In matrix sense, for $v_0 \in S_p M$, there exists an orthogonal matrix $p(v_0 )$ such that $p(v_0)^{-1}u^{+}(v_0)p(v_0)=u^{+}(-v_0)$.
    % Moreover, we can regard as $u^{+}(v_0)=u^{+}(-v_0)$
\end{Lem}

\begin{Thm}\label{thm:asymEin}
    If $(M,g)$ is an asymptotically harmonic manifold with minimal horospheres and Einstein, then $M$ is flat.
\begin{proof}
    % Consider $V(v)=u^{+}(v)-u^{-}(v)$. Since $\tr V(v)= 2h = 0$ for all $v\in SM$, from Lemma \ref{lem:nonnegativeV} we obtain $V(v)\equiv 0$, which implies that $u^{+}(v) = u^{-}(v)$ for all $v\in SM$. Using Lemma \ref{lem:u+}, for every point $p\in M$, there exists $v_0\in S_p M$ such that
    % \begin{equation*}
    %     u^{+}(-v_0) = u^{+}(v_0)= u^{-}(v_0)  = -u^{+}(-v_0).
    % \end{equation*}
    % Hence $u^{+}(-v_0)=0$ and $u^{+}(v_0)=0$. Then the Riccati equation \eqref{Riccati} yields $\Ric(v_0,v_0)=0$. Since $M$ is Einstein, $M$ is Ricci-flat, and therefore $M$ is flat.
       Consider $V(v)=u^{+}(v)-u^{-}(v)$. Since $\tr V(v)= 2h = 0$ for all $v\in SM$, from Lemma \ref{lem:nonnegativeV} we obtain $V(v)\equiv 0$, which implies that $u^{+}(v) = u^{-}(v)$ for all $v\in SM$. Using Lemma \ref{lem:u+}, for every point $p\in M$, there exists $v_0\in S_p M$ such that
    \begin{equation*}
        p(v_0)^{-1}u^{+}(v_0)p(v_0) = u^{+}(-v_0)= u^{-}(-v_0)  = -u^{+}(v_0).
    \end{equation*}
    It follows that $u^{+}(v_0)$ and $-u^{+}(v_0)$ have the same eigenvalues. Then the corresponding eigenvalues vanish and so $u^{+}(v_0)=0$. Hence, the Riccati equation \eqref{Riccati} yields $\Ric(v_0,v_0)=0$. Since $M$ is Einstein, $M$ is Ricci-flat, which implies that $M$ is flat.
    % It follows that $u^{+}(v_0)p(v_0) = -p(v_0)u^{+}(v_0)$, and from the fact that $\tr(u^{+}(v_0)p(v_0)) = - \tr(p(v_0)u^{+}(v_0))= - \tr(u^{+}(v_0)p(v_0)) $ we get $\tr(u^{+}(v_0)p(v_0))=0$. Also, $\det u^{+}(v_0) = (-1)^{n-1}\det u^{+}(v_0)$.
\end{proof}
\end{Thm}

Now we consider the case when $M$ is {a} symmetric space.
% We recall that $h$ denotes the mean curvature of horospheres of an asymptotically harmonic manifold.

\begin{Lem}[\cite{H}]\label{lem:sym}
    Let $(M,g)$ be an asymptotically harmonic manifold. Then $X\equiv 0$ if and only if $M$ is {a}  symmetric space.
\end{Lem}

\begin{Thm}
        If $(M,g)$ is an asymptotically harmonic with minimal horospheres {which is also a} symmetric space, then $M$ is flat.

        % when $h=0$, and $M$ is a real hyperbolic space of constant sectional curvature $-\frac{h^2}{(n-1)^2}$ when $h>0$.
    \begin{proof}
      Let us assume that $h=0$. Using $\tr V(v)= 2h = 0$ for all $v\in SM$ and Lemma \ref{lem:nonnegativeV}, we obtain $V(v)\equiv 0$, which means that
      % As an observation in the proof of Theorem \ref{thm:asymEin},
      $u^{+}(v)=u^{-}(v)$ for all $v\in SM$. Then we get $X(v)=-\frac{1}{2}(u^{+}(v)+u^{-}(v))=-u^{+}(v)$. Since $M$ is symmetric, from Lemma \ref{lem:sym} we have $u^{+}(v)=0$ for all $v\in SM$. Hence the Riccati equation \eqref{Riccati} gives us $K\equiv 0$.
      % Second, we assume that $h>0$. Since $M$ is symmetric, we have $X\equiv 0$ and $u^{+}(v) = -u^{-}(v)$ for all $v\in SM$. But $V'=XV+VX$, so $V(t)=V(\gamma_v'(t))=C$ for some constant $C$ along $\gamma_{v}$. Consider $v\ne w\in SM$. Then by the same reason, we have that $V(\gamma_w'(t))=C_1$ for some constant $C_1$ along $\gamma_w$. Taking the trace for both equations, we get $C=\frac{2h}{n-1}=C_1$. Hence, we have $V(v)=\frac{2h}{n-1}$ for all $v\in SM$ and so $u^{+}(v)=\frac{h}{n-1}\id$ for all $v\in SM$. By the Riccati equation \eqref{Riccati}, $R(0) = R(\gamma_v'(0))= R(x,v)v = -\frac{h^2}{(n-1)^2}\id$ for all $v\in SM$ and $x\in v^\perp$. Therefore, $M$ is a real hyperbolic space of constant sectional curvature $-\frac{h^2}{(n-1)^2}$.
    \end{proof}
\end{Thm}

\vspace{0.2in}

{On the other hand, it is known that the following relation holds for a Riemannian manifold $M$:}
{
\begin{center}
    $M$ has nonpositive sectional curvature, then
    $M$ has no focal points which implies that $M$ has no conjugate points.
\end{center} }

For an asymptotically harmonic manifold, if we strengthen the condition
 of ``without conjugate points", then we obtain that an asymptotically harmonic manifold with minimal horospheres is flat as follows: If $u^{+}(v)$ is positive semi-definite and $\tr u^{+}(v)=0$, $u^{+}(v)=0$. Then the Riccati equation \eqref{Riccati} implies $K \equiv 0$.

\begin{Pro}
    Let $(M,g)$ be a complete, simply connected manifold with nonpositive sectional curvature {\rm(}i.e. Hadamard manifold{\rm)}. If $M$ is an asymptotically harmonic with minimal horospheres, then $M$ is flat.
    \begin{proof}
        As $M$ has nonpositive sectional curvature, using standard comparison geometry we get $u^{+}(v)$ is positive semi-definite. Then by the above argument, we have that $M$ is flat.
        % Hence $\tr u^{+}(v)=h=0$ implies $u^{+}(v)=0$. Thus the Riccati equation \eqref{Riccati} yields $R\equiv 0$.
    \end{proof}
\end{Pro}

\begin{Pro}[\cite{Z}]
     Let $(M,g)$ be a complete, simply connected Riemannian
manifold without focal points. If $M$ is an asymptotically harmonic with minimal horospheres, then $M$ is flat. Moreover, asymptotically harmonic manifolds without focal points which is a Riemannian product must be flat.
\end{Pro}

Besides, when $M$ has no conjugate points, in \cite{RS2} the authors obtained the following.
\begin{Pro}[\cite{RS2}]
    Let $(M,g)$ be a complete, simply connected Riemannian
manifold without conjugate points.  If $M$ is an asymptotically harmonic with minimal horospheres such that
{the determinant of the second fundamental form of geodesic spheres is a radial function}, then $M$ is flat.
\end{Pro}

% \begin{Rem}
%     In \cite[Lemma~35]{Z}, it was shown that for an irreducible symmetric space of noncompact type, it is asymptotically harmonic if and only if it is of rank-one.
% \end{Rem}

% \noindent
% {\bf  Final Remark:} We expect the same result in higher dimensions. But the
% whole proof relies heavily on the assumption that the dimension of $M$ is $3$.

\medskip
At the end of the paper, we state the following conjecture which is analog of Lichnerowicz's conjecture
for asymptotically harmonic manifolds. This conjecture has been partially resolved to date. Including Theorem~\ref{class}, we refer to \cite{CS, He, H, KP, S, Z}.

\begin{Con*}
    Let $(M,g)$ be a complete, simply connected Riemannian
manifold without conjugate points. If $M$ is an asymptotically harmonic manifold, then $M$ is either flat or rank-one symmetric space of noncompact type.
\end{Con*}

\section*{Acknowledgements}
This work was supported by the National Research Foundation of Korea (NRF) grant funded by the Korea government (MSIT) (NRF-2019R1A2C1083957). The authors thank to Paul-Andi Nagy for the useful discussion on orthogonal coordinates.

\end{document}